 \documentclass[11pt]{amsart}
\usepackage{amsmath,amssymb}
\usepackage{graphicx}
\usepackage{amsfonts,amscd,latexsym}

\newcommand{\labell}[1] {\label{#1}}

\newcommand{\NN}{{\mathbb N}}

\newcommand{\less}{{\smallsetminus}}

\newcommand{\bla}{{\bigl\langle}}
\newcommand{\bra}{{\bigl\rangle}}
\newcommand{\un}{\underline}

\newcommand{\p}{{\partial}}
\newcommand{\al}{{\alpha}}

\newcommand{\be}{{\beta}}

\newcommand{\Om}{{\Omega}}
\newcommand{\om}{{\omega}}
\newcommand{\eps}{{\varepsilon}}
\newcommand{\de}{{\delta}}

\newcommand{\ga}{{\gamma}}
\newcommand{\Ga}{{\Gamma}}

\newcommand{\la}{{\lambda}}

\newcommand{\si}{{\sigma}}
\newcommand{\Si}{{\Sigma}}

\newcommand{\Vv}{{\mathcal V}}
\newcommand{\Ee}{{\mathcal E}}

\newcommand{\Ww}{{\mathcal W}}
\newcommand{\Oo}{{\mathcal O}}
\newcommand{\Mm}{{\mathcal M}}
\newcommand{\Ss}{{\mathcal S}}
\newcommand{\oMm}{{\overline {\Mm}}}
\newcommand{\ov}{\overline}

\newcommand{\PP}{{\mathbb P}}
\newcommand{\Q}{{\mathbb Q}}

\newcommand{\C}{{\mathbb C}}

\newcommand{\SSS}{{\smallskip}}

\newtheorem{theorem}{Theorem}[section]

\newtheorem{lemma}[theorem]{Lemma}

\newtheorem{prop}[theorem]{Proposition}

\newtheorem{example}[theorem]{Example}

\newtheorem{rmk}[theorem]{Remark}

\numberwithin{figure}{section}
\numberwithin{equation}{section}
\numberwithin{table}{section}

\newcommand{\MS}{{\medskip}}

\newcommand{\NI}{{\noindent}}

 \title{Comparing absolute and relative Gromov--Witten invariants}
 \author{Dusa McDuff}\thanks{partially supported by NSF grant DMS 0604769}
\address{Department of Mathematics,
Barnard College, Columbia University, 
NY 10027, USA}
\email{dusa@cpw.math.columbia.edu}
\date{August 4, 2008}
\begin{document}

\begin{abstract} This  note
compares the usual (absolute) Gromov-Witten invariants of a symplectic manifold 
with the  invariants that count the curves relative to a (symplectic)
divisor $D$.   We give explicit 
examples where these invariants differ even though
it seems at first that they should agree, for example
when counting genus zero curves in a class 
$\be$ such that $\be\cdot D=0$.   
The main tool is the decomposition formula in the form
developed by A. Li--Ruan.   
\end{abstract}

\maketitle

\tableofcontents

\section{Introduction}

This is a largely expository paper about 
the An Li--Ruan version of relative 
Gromov--Witten invariants;
 cf. \cite{LR}.
These  have been used 
for a long time to help solve enumerative problems 
(cf. Ionel--Parker~\cite{IP}  for example)  but have only recently been much used in other areas of symplectic geometry.
For example, they are an essential ingredient in Hu--Li--Ruan's work 
in ~\cite{HLR} on the symplectic birational invariance of the uniruled property as well as in later papers 
\cite{HR,Lai,TJLR,Mcu} on related topics.
They are still rather little understood by symplectic geometers, even in the genus zero case.

\S1 describes the moduli spaces of stable maps that 
enter into their definition.  These involve curves on several levels as in many SFT moduli spaces. (Of course, relative Gromov--Witten invariants can be considered as a special case of SFT; cf. Bourgeois {\it et al.} \cite[Example~2.2]{BEH}.)  We do not attempt to do any analysis, but rather to explain the structure of these moduli spaces through examples.   \S2 and \S3 describe the decomposition formula and give several examples of its use.  

Here are our main results.  
\MS

\NI {\bf 1.}
We show in Remark~\ref{rmk:g0} that when the 
normal bundle of the divisor $D$ is suitably positive one can 
calculate genus zero Gromov--Witten invariants without involving 
higher level curves.  (This result  is greatly strengthened by Zinger in \cite{Z3}.)  However, in genus one this need not be the case; cf. Example \ref{ex:rel2}.\SSS

\NI {\bf 2.} Our second set of results concern the contributions to the relative invariants from curves in or near $D$.  More formally, these are the relative invariants of the pair $(Y,D^+)$, where  
$Y: =  \PP(L\oplus \C)$ is the projectivization of  
a complex line bundle $L\to D$ and 
$D^+: = \PP(L\oplus \{0\})$ is the copy of $D$ \lq\lq at infinity" with normal bundle $L^*$.  We denote by 
$D': = \PP(\{0\}\oplus \C)$ the copy of $D$ along the zero section
(with normal bundle $L$).
 
 Understanding the relative invariants of $(Y,D^+)$ is a crucial ingredient of the proofs of many results, such as the blow up formulas in \cite{HLR, Lai}.  As a first step, one needs to know that certain invariants vanish.  
 Lemma~\ref{le:relD} gives a version of this result for
relative invariants in class $\be\in H_2(Y)$ under the condition that
$\be\cdot D'\ge 0$ and the homological constraints are pulled back from $D$.  The argument is elementary, based on a  dimension count.
This argument fails when $\be\cdot D'<0$, but there were no explicit examples showing that the result may also not hold in this case.
Example \ref{ex:relD} gives such an example in the genus zero case.    
\SSS

\NI {\bf 3.}  Our first application of the decomposition formula is to give conditions under which the absolute and relative genus zero invariants agree for classes $\be$ such that $\be\cdot D=0$: see 
Proposition \ref{prop:g0}.  Its statement  is 
slightly stronger than other similar results in the literature (cf. \cite{HR, Lai})
 because  we do not make the usual positivity assumptions on the
 normal bundle of $D$.  However the argument is still based on a  dimension count.
 \SSS

\NI {\bf 4.}  Finally we give examples 
where the absolute and relative invariants differ for classes $\be$ such that $\be\cdot D=0$.  It is easy to find such examples in genus $>0$; see Example \ref{ex:g1}.  The genus zero example
is trickier and is given in Proposition~\ref{prop:=0}.

The reason why such examples exist is that in calculating the relative invariants one  uses almost complex structures $J$ and perturbing $1$-forms $\nu$ that respect the divisor $D$.   If $D$ is \lq\lq nonnegative" then a generic pair $(J,\nu)$ of this form should 
satisfy the regularity conditions for stable maps to $X$ as 
 as well as those for maps to $(X,D)$.  (This statement is made precise in Zinger \cite{Z3}.)
However, if $D$ is sufficiently \lq \lq negative" then 
this is not so, and there may be contributions to the relative invariant that are perturbed away in  the absolute invariant.

%
%


%

%

%


%


\MS

\NI {\bf Acknowledgements}  I wish to thank the many people who have helped me understand relative invariants, in particular T.-J. Li, Yongbin Ruan, Rahul Pandharipande, and Aleksey Zinger.




\section{Relative Gromov--Witten invariants}

We first explain the structure of the stable maps used 
to calculate the relative invariants, illustrating by many examples. Then we describe some elementary  vanishing results.

\subsection{Relative stable maps}

Consider a pair $(X,D)$ where $D$ is a divisor in $X$, i.e. 
 a codimension $2$ symplectic submanifold.
The relative invariants  count connected $J$-holomorphic curves in some $k$-fold prolongation $X_k$ of $X$. 
Here $J$ is an $\om$-tame almost complex structure on $X$ satisfying certain normalization conditions along $D$.  In particular $D$ is $J$-holomorphic, i.e. $J(TD)\subset TD$. The invariants are defined by first 
constructing a compact moduli space $\oMm: = \oMm\,\!^{X,D}_{\be,\un d}(J)$ of genus $g$ $J$-holomorphic
curves $C$ in class $\be\in H_2(X)$ as described below and then integrating the given constraints over the corresponding virtual cycle $\oMm\,\!^{[vir]}$. Here $\un d = (d_1,\dots,d_r)$ is a partition of $d: =
 \be\cdot D\ge 0$.  
 
 To a first approximation the moduli space consists of curves, i.e. equivalence classes $C=[\Si,u,\dots]$ of stable maps, that intersect the divisor $D$ at $r$ points with multiplicities $d_i$.  
More precisely, each such  curve has $k+1$ 
levels $C_i$ for some $k\ge 0$, the principal level 
$C_0$  in $X\less D$ and the 
higher levels $C_i$ (sometimes called bubbles) in  the $\C^*$-bundle
$L_D^*\less D$ where $L_D^*\to D$ is the dual of the normal bundle to $D$. 
We identify $L_D^*\less D$ with the complement of the sections $D_0,D_\infty$ of the ruled surface 
$$
\pi_Q: Q: = \PP(\C\oplus L_D^*)\to D,
$$
where the zero section $D_0: = \PP(\C\oplus\{0\})$ of $L_D^*$ has normal bundle $L_D^*$ and  the infinity 
section $D_\infty: = \PP(\{0\}\oplus L_D^*)$ has normal bundle $L_D$.
We then think of the original divisor $D\subset X$ as the infinity section $D_{0\infty}$ at level $0$. For each $i>0$
the  ends of $C$ along the zero section $D_{i0}$ of the components at  level $i$ match with those along $D_{i-1\,\infty}$ of the preceding level, and the relative constraints are put along the infinity section 
 $D_{k\infty}$ of  the last  copy of $L_D^*$.   
  
Each $C_i$ has three disjoint (and possibly empty) sets of marked points, 
the absolute marked points, the points $\{y_{j0}: 1\le j\le r_{i0}\}$ where it meets $D_{i0}$ and the points $\{y_{j\infty}: 1\le j\le r_{i\infty}\}$ where it meets $D_{i\infty}$.  (For this to make sense when $i=0$ take
$D_{00}$ to be empty.)
We require further that each component $C_i$   be stable, i.e. have a finite group of automorphisms.  For $C_0$ this has the usual meaning.  However, when 
  $i\ge 1$ we identify two level $i$ curves $C_i, C_i'$ if they lie in the same orbit of the fiberwise $\C^*$ action; i.e.  given representing maps $u_i:(\Si_i,j)\to Q$
 and 
$u_i':(\Si_i',j')\to Q$,  the curves are identified if there is $c\in \C^*$ and a holomorphic map $h:(\Si_i,j)\to (\Si_i',j')$ (preserving all marked points) such that $u_i\circ h=c \,u_i'$.\footnote
 {
Hence a level $C_i$ cannot consist only of multiply covered curves
$z\mapsto z^k$ in the fibers $\C^*$ of $L_D^*\less D$ because these are not stable.} 
Thus $L_D^*\less D$ should be thought of as a \lq\lq rubber" space; cf. \cite{MP}.

  Note that, although the whole curve is connected, 
the individual levels need not be but
should fit together to form a genus $g$ curve.  
The homology class $\be$ of such a curve is defined to be the sum of the homology class of its principal component with
the projections to $D$ of the classes of its higher levels.
\footnote
{
This is the class of the curve in $X$ obtained by gluing all the levels together.
When $k=0$ the curve only has one level and so the relative constraints lie along $D\subset X$.  These are the curves counted by Ionel--Parker~\cite{IP}. }  
When doing the analysis it is best to think that the domains and targets of the curves have cylindrical ends.  However, their indices
are the same as those of the corresponding compactified curves; see \cite[Prop~5.3]{LR}.  
 Thus the (complex) dimension of the moduli space $\oMm\,\!^{X,D}_{g,\be,k,(d_1,\dots,d_r)}$ of genus $g$ curves in class $\be$ with $k$ absolute marked points and $r$ relative ones is
\begin{equation}\labell{eq:dim}
 n(1-g)  + c_1^X(\be) + k+r + 3(g-1) - \sum_{i=1}^r(d_i-1)
 \end{equation}
Here 
we subtract $d_i-1$
at each relative intersection point of multiplicity $d_i$ 
 since, as far as a dimensional count is concerned, what is happened at such a point is that $d_i$ of the $d: = \be\cdot D$ intersection points of the $\be$-curve with $D$ coincide.

\begin{example}\labell{ex:rel}\rm (i)  Let $X = \PP^2\#{\ov \PP}\,\!^2$, the one point blow up of $\PP^2$ and set $D: = E$, the exceptional curve. Denote by $\pi:X\to \PP^1$ the projection, and fix another section $H$ of $\pi$ that is disjoint from $E$. Let $J$ be the usual complex structure and $\oMm\,\!^{X,E}_\la(J;p)$ be the moduli space of holomorphic lines through some point $p\in H$, where $\la=[H]$ 
is the class of a line.
Since $\la\cdot E=0$ there are no relative constraints.
Then $\oMm\,\!^{X,E}_\la(J;p)$ has complex dimension $1$ and should be diffeomorphic to $\PP^1$.  The closure of the ordinary moduli space 
of lines in $X$ though $p$ contains
all such lines together with one reducible curve consisting of the union of the exceptional divisor $E$ with the fiber $\pi^{-1}(\pi(p))$ through $p$.  But the elements of $\oMm\,\!^{X,E}_\la(J;p)$ do not contain components in $E$.  Instead, this component becomes a higher level curve lying in
$Q=\PP(\C\oplus \Oo(1))$ that intersects 
$E_0=\PP(\C\oplus \{0\})$ in the point $E_0\cap \pi^{-1}(\pi(p))$ and lies in the class $\la_Q$ of the line in $Q\cong X$.  Note that modulo the action of $\C^*$ on $Q$ there is a {\it unique} such bubble.  Thus the corresponding two-level curve in $\oMm\,\!^{X,E}_\la(J;p)$ is a rigid object.
Moreover, because $\la_Q$ projects to the class $\eps\in H_2(E)$ of the exceptional divisor, its homology class $(\la -\eps) + pr(\la_Q)$ is $
(\la -\eps) + \eps = \la.$  Such a bubble can also be thought of as a holomorphic section of the bundle $\Oo(1)$ that is zero at $\pi(p)$.
\SSS

\NI (ii)  If $X$ is as above but $D$ is the line $H$ then
$\oMm\,\!^{X,H}_{\la,1}(J)$ has complex dimension $2$.  It contains a 
$1$-dimensional family of curves that each consist entirely of a level $1$ bubble, with one such curve for each point in $H_{1\infty}$. 
\end{example}

To state the index formula we need more notation.  For simplicity let us suppose there are no absolute marked points.
Then each level $C_i$ (which may not be connected)  is an equivalence class of stable maps 
 $$
 [\Si_i,u_i, y_{10},\dots,y_{r_{i0}0},y_{1\infty},\dots,
 y_{r_{i\infty }\infty}],
 $$
in some class $\be_i$,  where the internal relative marked points
 $y_{10},\dots,y_{r_{i0}0}$ are mapped to $D_{i0}$ 
 with multiplicities
  $\un m_{i0}$  that sum to $\be_i\cdot D_{i0}$
  and $y_{1\infty },\dots,y_{r_{i\infty }\infty}$ are taken to $D_{i\infty}$ with multiplicities $\un m_{i\infty }$ that sum to $\be_i\cdot D_{i\infty }$.  
 The index $I_i$ of such a curve is defined to be the 
formal dimension of the moduli space of all such curves with domain of the given genus, where we impose the conditions that $u_i(y_{\ell0})\in D_{i0}, 
u_i(y_{\ell \infty})\in D_{i\infty }$, subtract $-1$ (because of the $\C^*$ action) when $i>0$, and also subtract 
$(\be_i\cdot D_{i0} - r_{i0}) + (\be_i\cdot D_{i\infty i} - r_{i\infty })$ 
to take care of the multiplicity of the intersections 
along the divisors.  
The components $C_i$ and $C_{i+1}$ match along $D_{i\infty } = 
D_{i+1,0}$ only if the multiplicities $\un m_{i\infty }$ and
$\un m_{i+1,0}$ agree (after possible reordering). Also
we need corresponding ends of $C_i$ and $C_{i+1}$ to intersect transversally in $D_{i\infty } = D_{i+1,0}$ (so that they can be glued).  
Thus the index (over $\C$) of the $(k+1)$-level curve
 $C=\cup_{i=0}^{k} C_i$ is 
\begin{equation}\labell{eq:kind1}
 {\rm ind\,} C = \sum_{i=0}^{k} I_i - r(n-1),
\end{equation}
  where $I_i$ is as above and $r = \sum_{i=1}^k r_i$ is the number of internal intersection points.  

 The next result was noted by Ionel--Parker~\cite[Prop.~14.10]{IP}.
 
 \begin{lemma}\labell{le:kind1} 
 In the genus zero case, the index of a higher level component in $Q$ equals that of its projection to $D$ (considered as a curve in $D$).
\end{lemma}

\begin{proof}  Consider a  sphere
 in $Q$  in a class $\be: = \al+df$, where $f$ is the class of the fiber of $\pi_Q:Q\to D$ and $\al\in H_2(D_0)$,  
that intersects $D_\infty$ at $r_\infty$ points 
 through cycles $b_\infty: = (b_{\infty i})$ with multiplicities $d_{\infty i}$ 
 summing to $d$, and  intersects $D_0$ at $r_0$ points 
 through cycles $b_0: = (b_{0j})$ with multiplicities $d_{0j}$ 
  summing to $d + \al\cdot D_0$. Let $\de_b$ be half the sum of the degrees of the cycles $b_{\infty i}, b_{0j}$ and set $r: = r_\infty + r_0$, the total number of constraints. Then the formal (complex) dimension of the moduli space of such spheres (when quotiented out by the $\C^*$ action on $Q$) is
\begin{eqnarray*}
&&  n + \bigl(c_1^D(\al) + \al\cdot D_0 + 2d\bigr) + (r -3)  - (d-r_\infty) - (d+\al\cdot D_0 - r_0) + (\de_b - rn) -1 \\
 && \qquad\qquad = \;\;
  n + c_1^D(\al) + 2r-3 + \de_b - rn -1.
\end{eqnarray*}
Here $n: = \dim_\C X$,
 $ c_1^D(\al) + \al\cdot D_0 +2d=c_1^Q(\be)$,
 the term $r-3$ takes care of the variation of the marked points,  the terms $d-r_\infty, d+\al\cdot D_0 - r_0$ 
appear because we must subtract $d_i-1$
 at each relative insertion of order $d_i>1$, and we subtract
 $rn-\de_b$ to take care of the homological constraints.
 
When we project these spheres to $D$ they still have $r$ marked points through the constraints  $b_{\infty, i}, b_{0,j}$, now considered as lying in $D$, but all information about the orders of the normal tangencies is lost.  The formal dimension of the moduli space of such curves is
$(n-1) +  c_1^D(\al) + r-3 + \de_b - r(n-1)
$, 
 just what we found above.  \end{proof}
The next result can be easily checked; it is implicit in Li--Ruan~\cite{LR} formula (5.1).

\begin{lemma}\labell{le:kind}  The index of a genus $g$ 
curve in class $\be$ 
with $k+1$ levels  and given incidence multiplicities 
along $D_{\infty k}$ is $I(g,\be) - k$, where 
$I(g,\be)$ is the index of the main stratum of curves with no higher levels.
\end{lemma} 

The constraints for the relative invariants consist of homology classes  $b_j$ in the
divisor $D$ (the relative insertions) together with absolute (descendent) insertions 
$\tau_{i_j}a_j$ where $a_j\in H_*(X)$ and $i_j\ge 0$.
   We shall denote
the (connected) relative genus zero invariants by:
\begin{equation}\labell{eq:GW}
\bla \tau_{i_1}a_1,\dots,\tau_{i_q}a_q\,|\, b_1,\dots,b_r\bra^{X,D}_{0,\be, (d_1,\dots,d_r)},
\end{equation}
where $ a_i\in H_*(X), b_i\in H_*(D),$ and $ d: = \sum d_i = \be\cdot D\ge 0.
$  (If $\be\cdot D< 0$ then the moduli space is undefined and the corresponding invariants are set equal to zero.)
This invariant counts connected genus zero curves in class $\be$
that go through cycles representing the classes $a_i$ and intersect $D$ to order $\un d: = (d_1,\dots,d_r)$ in the $b_i$,  i.e the $i$th relative marked point intersects $D$ to order $d_i\ge 1$ at a point on some representing cycle for $b_i$.  We assume that the cycles $a_i$ are in general position to $D$, while the $b_i$ lie in $D$.  Moreover, the insertion
$\tau_{i}$ occurring at the $j$th absolute marked point means that
we add the constraint $(c_j)^i$, where $c_j$ is the first Chern class of the cotangent bundle to the domain at the $j$th marked point.    One can evaluate
(\ref{eq:GW}) (which in general is a rational number) by integrating
an appropriate product of Chern classes over the virtual cycle corresponding to
the moduli space of stable $J$-holomorphic  maps that satisfy the given homological constraints and tangency conditions.  For details on how to construct this virtual cycle  see for example Li--Ruan~\cite{LR} or Hu--Li--Ruan~\cite{HLR}. \footnote
{
What one needs here is an appropriate framework in which to construct suitable multivalued perturbations as in \cite{Mcbr}; see also \cite{Mcq}.  There are many possible ways of solving this problem.
A  very general construction that applies in a wide variety of situations will eventually
 be provided by Hofer--Wysocki--Zehnder~\cite{HWZ} and \cite{H}.  This should give a coherent setting in which one could reprove the results claimed here. For another approach see Zinger \cite{Z2,Z3}.}
 Since the virtual cycle  has dimension equal to the relevant
  index, the above formula defines
 a number  only when this index equals the sum of the dimensions of the constraints.  In all other cases the invariant is defined to be zero.
\MS

We now give some examples to explain the structure of the stable maps considered here, and to 
illustrate the index calculations.
 
\begin{example}\labell{ex:relcon}\rm   (i)
Let $(X,D) = (\PP^2,\PP^1)$ and consider the moduli space 
$\oMm_{2\la,(2),4}^{X,D}(p)=:\oMm$ of degree $2$ stable maps,  with $3$ absolute marked points $z_1, z_2, z_3$ and one relative marked point $z_0$, and
that are tangent to $D$ at the fixed point $p=u(z_0)\in D$.  Its open stratum $\Mm$
consists of equivalence classes $[u;z_0,\dots,z_3]$ where
 $u:\PP^1\to \PP^2$ is a map of degree $2$ (i.e. a conic) such that $u^{-1}(p) = 
 u^{-1}(D)=z_0$.
The (complex) dimension of $\Mm$ is $6$, one less than the dimension of the space of such maps with $u(z_0)=p$, since the tangency condition can be interpreted as saying that the two intersection points of the conic with $D$ coincide.  The evaluation map takes this open stratum injectively onto the set of generic triples of points $u(z_1),u(z_2), u(z_3)$ in $\PP^2$.

The full moduli space $\oMm$
contains various degenerate curves.   We now investigate some of
 the corresponding strata, showing that their dimensions
agree with the formula in Lemma~\ref{le:kind}. The usual stratification of $\oMm$ (which is determined by 
 the domain and marked points of the stable map) is 
  closely related to the positions of the image points $u(z_1), 
  u(z_2), u(z_3)$, and we shall analyze $\oMm$ by looking at this image configuration rather than the strata.

If  $u(z_1), u(z_2), u(z_3)$ are collinear the curve splits into a line $\ell$ in $\PP^2$ through these three points  together with a bubble component in 
$Q: = \PP(\C\oplus L_D^*)$ 
 through the point $q:=\ell\cap D$
on the zero section $D_{0}: = \C\oplus \{0\}$ of $Q$ and tangent to the infinity section $D_{\infty}$ at the point $u(z_0)$.   One can think of this bubble as a meromorphic section of $L_D^*$ with a  
pole of order $2$ at $p$ and a simple zero at $q$, or  from a more symplectic viewpoint  as a curve in the compact ruled surface $Q$
 in the class $D_0 + 2f$, where $f$ denotes the class of the fiber of $Q\to D$.  It is  counted modulo the $\C^*$ action that fixes $0$ and the points of $D_\infty$.  Since $Q$ is the one point blow up of $\PP^2$ with exceptional divisor $D_0$, each such bubble 
corresponds to a conic on $\PP^2$ (the blow down of $Q$) that is tangent to $D_\infty$ at a fixed point on $D_\infty$ and also tangent to a fixed direction at $0$.  There is one such conic modulo the $\C^*$ action.  Hence this bubble is rigid (i.e. lies in a zero-dimensional moduli space)
  and
contributes to the nonzero rubber invariant 
$$
\bla pt |\;|pt\bra^{D_0,Q,D_\infty,\sim}_{0,2f+\la_0,(1),(2)}
$$
that counts curves with these constraints modulo the $\C^*$ action on the fibers of $Q\to D$.   (Here the relative constraints are put  
on the left and right of the bracket and we use $\sim$ 
to denote rubber invariants as in Maulik--Pandharipande~\cite{MP}.)  
Since the (collinear)  points $u(z_1),u(z_2),u(z_3)$ have $5$ degrees of freedom, these curves form a codimension $1$ subspace of the full moduli space.  This is consistent with 
Lemma~\ref{le:kind}; these curves are regular and so should form a space whose dimension equals the index of the given stratum.

A second kind of degeneration occurs when  
$u(z_1)$ and $ u(z_2)$ lie on a line through the fixed point $p$. 
 Now the principal part of the curve has two components, one a line through $u(z_1), u(z_2)$ and the other a line through $u(z_3)$ and $p$, that are
 joined by a level $1$ bubble in class $2f$ that lies in the fiber over $p$, intersecting
 $D_{10}$ twice at $p$ with order $1$ and intersecting $D_{1\infty}$ once  with order $2$; see Figure \ref{fig:GW1}.  This fiberwise 
 bubble is rigid, and
 contributes to the nonzero rubber invariant 
$$
\bla D_0, D_0 |\;\;|pt\bra^{D_0,Q,D_\infty,\sim}_{0,2f,(1,1),(2)}.
$$

\begin{figure}[htbp] 
   \centering
   \includegraphics[width=2.2in]{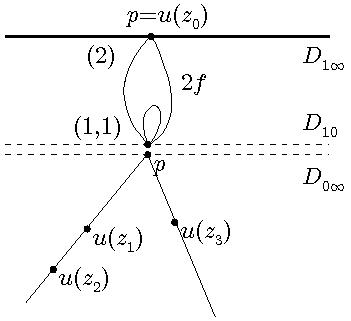} 
   \caption{The two component stable map 
   in Example \ref{ex:relcon} (i) when $p, u(z_1)$ and $u(z_2)$ are collinear. The schematic picture of the bubble is meant to illustrate its parametrization; its image lies entirely in the fiber over $p$ and it intersects the divisors $D_{10}, D_{1\infty}$
   with multiplicities $(1,1), (2)$.}
   \label{fig:GW1}
\end{figure}

\NI
Because this rubber invariant counts curves in a fiber class that go through a fixed point,
it equals the corresponding 
invariant in the fiber  of $Q\to D$. Identify this fiber with
$S^2:=\C\cup \{\infty\}$, meeting 
 $D_0, D_\infty$ in the points $\{0\}$ and $\infty$.  Then
this invariant  counts the maps $\C\to \C$ of the form $z\mapsto a z(z-b)$ modulo the reparametrization $z\mapsto cz$ of the domain
and the $\C^*$ action $w\mapsto dw$ on the target.  Hence this invariant is $1$.  Again the points $u(z_1), u(z_2), u(z_3)$ have $5$ degrees of freedom, so that this stratum of regular curves has codimension $1$.

If all $4$ points are collinear (but still distinct), then we could  have: 
\SSS

\NI
(a) two distinct 
lines through
$p$ together with a bubble at $p$ as in the preceding paragraph but with $z_1,z_2,z_3$ all on the same line,

\NI
(b) two coincident lines plus a bubble (with various possibilities  for the marked points), or 

\NI
(c)
 a doubly covered line through these $4$ points with one of its branch points at $p$.  
 \SSS
 
 Since the collinearity condition has index $2$ we would expect 
 the corresponding spaces of stable maps to have
codimension $2$.  This is true in case (b), because although 
the curves are regular and have two levels their principal components satisfy an extra constraint.  However  in  cases (a) and  (c)
the curves  have 
$5$ degrees of freedom; e.g. in (c) there is one for the line, one for its second branch point and three for the positions of the other points on it. 
In these degenerate cases the curves are not uniquely 
determined by the position 
of the points $u(z_1), u(z_2), u(z_3)$, more precisely
the collinearity conditions are not transverse to the evaluation map
$\oMm\to (\PP^2)^3$ and so do not cut down its dimension in the expected way. 

Other degenerations occur when some of the points coincide; we leave 
further discussion to the reader.
\SSS

\NI (ii)
Now  consider the moduli space $\oMm$ of conics in $\PP^2$ with $3$ absolute  and $2$ relative marked points that intersect the line $D$ with multiplicity $\un 2=(1,1)$.  This has dimension $8$.  
 $\oMm$
contains a codimension $1$ stratum $\Ss$ consisting of curves whose two intersection points with $D$ coincide.     
 The open stratum in $\Ss$ is the $7$-dimensional space of two-level curves where $C_0$ 
is tangent to $D_{0\infty}$ at some unspecified point and $C_1$ is a rigid fiberwise bubble contributing to the rubber invariant 
$$
\bla pt |\;\;|D_\infty, D_\infty\bra^{D_0,Q,D_\infty,\sim}_{0,2f, (2), (1,1)}.
$$
In other words, $C_1$ is tangent to $D_{10}$ at some specified point and intersects $D_{1\infty}$ 
twice transversally at arbitrary points.  Thus
 the incidence conditions of the bubble along $D_{1\infty}$ have type $(1,1)$ while the incidence conditions along the internal divisor
 $D_{10}$
 are dual to those of the principal component. Note that these are not the same as the elements considered in (i) because the 
 incidence conditions along $D$ 
 have type $(1,1)$ instead of $(2)$.
 \end{example} 
 
\begin{figure}[htbp] 
   \centering
   \includegraphics[width=4in]{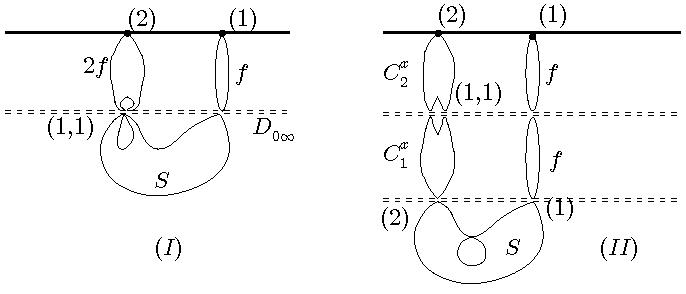} 
   \caption{Some of the stable maps considered 
   in Example \ref{ex:rel2}.   In (I) $S$ is an immersed sphere of 
   degree $3$ with its node on $D_{0\infty}$ while in (II) it
   is tangent to $D$.
   Each bubble lies in a single fiber. Hence when a bubble intersects one of the  divisors $D_{ij}$ with multiplicity $(1,1)$ its two intersection points with $D_{ij}$ coincide.}
   \label{fig:GW2}
\end{figure}

\begin{example}\labell{ex:rel2}\rm   Consider the moduli space of degree $3$ tori in $X: = \PP^2$
that intersect $D:=\PP^1$ with multiplicity $(2,1)$.  Its main stratum 
consists of tori tangent to $D$ and by formula (\ref{eq:dim}) has dimension $8$.  
Figure \ref{fig:GW2} illustrates two ways in which
 the genus of the stable map may carried by the pattern of intersections of the bubble components, rather than by the genus of any single component. 
  
In (I), the
main component $S$ is an immersed sphere of degree $3$ that
intersects $D$ transversally at its node as well as at one other point. The level $1$ curve consists of two fiberwise bubbles, 
one of which is 
doubly covered and intersects $D_{10}$ with multiplicity $(1,1)$, thus creating some genus.  Since both bubbles are rigid, the dimension of the space of curves of this type is the same as the dimension of the space of spheres $S$ with node on $D$; thus it is $7$, which agrees with the formula in Lemma~\ref{le:kind}.

In (II) the main component 
is still an immersed sphere of degree $3$, but now we assume that it is tangent to $D$ at some point $x$, and has one other transverse intersection point with $D$.  In particular, its node is disjoint 
from $D$.
There are two other levels that
contain   
doubly covered spheres $C_1^x, C_2^x$
in the fiber over $x$ that are tangent to $D_{10}$ and $D_{2\infty}$ but intersect the intermediate divisors $D_{1\infty}, D_{20}$ with multiplicity $(1,1)$. (Each of these levels also contains a singly covered sphere through the other intersection with $D$, but these are irrelevant to
 the present discussion.)
  The dimension of this  stratum is again equal to the dimension of the family of spheres $C_0$, and so is $7$, which is not equal to the index. The point is that although
 the components $C_i^x$ are stable and regular, they lie in moduli spaces that do {\it not} intersect transversally along $D_{1\infty}= D_{20}$ because their two intersections with this divisor must always coincide.  Therefore the only way to regularize this stratum is by gluing. The dimensions then work out correctly, because 
one must subtract $1$, the dimension of the relevant
obstruction bundle.
One way to understand this is to note that the two bubbles $C_1^x$ and $C_2^x$ do glue to give a family of regular tori in the 
single $\PP^1$ fiber over $x\in D$.  However this torus has trivial normal bundle when considered as a curve in $Q$
and so is {\it not} regular in $Q$ but rather has a cokernel of dimension $\dim D$.
 This remark also explains why there is no similar problem with the index in case (I) since in this case the glued curve lies in $X$. 
\end{example}
 
The above nontransversality phenomenon does not occur in genus zero.
As we now explain, it follows that as long as
 the divisor $D$ is suitably positive we can ignore 
 the curves with higher levels.   
 
 \begin{rmk}\labell{rmk:g0}\rm  Suppose that $D$ has \lq\lq nonnegative" normal bundle $L_D$; for example we might
 suppose that
 $c_1(L_D) = \la [\om|_D]$ for some $\la\ge 0$, or, more generally that $c_1(L_D)(\al)\ge 0$ for all classes $\al\in H_2(D)$ that can be represented by $J_D$-holomorphic spheres. 
 Then
in the genus zero case, {\it curves with higher levels do not contribute to  relative invariants}
\begin{equation}\labell{eq:Gw0}
\bla a_1,\dots,a_q\,|\, b_1,\dots,b_r\bra^{X,D}_{0,\be, (d_1,\dots,d_r)}
\end{equation}
{\it with only homological constraints.}
By this we mean the following.  
Denote by $\oMm\,\!^{X,D}_{\be,\un d}(J)$ the moduli space of all genus $0$ stable $J$-holomorphic curves in class $\be$ with $q$ absolute and $r$ relative marked points.  Suppose we partially regularize this moduli space by perturbing the equations for the components of the curves in each level to make them as generic as possible,  but not performing any gluing operations.
Here we only allow perturbations $\nu$ that preserve the structure of the target; i.e. we look at equations of the form $\ov{\p}_J u = \nu(u)$ where the perturbing $1$-form $\nu(u)$  at level $0$ vanishes 
sufficiently fast in the normal directions to $D$ and at higher levels is pulled back from $D$ by the projection $\pi:Q\to D$. 
Hence each higher level component still projects to a curve in $D$ and 
our procedure regularizes each of these projected components as a curve in $D$. 

We claim that in the genus $0$ case this procedure is enough to reduce the dimension of the set $\Ss_H$ of curves with higher levels 
below that of the full  moduli 
space.  Hence the partially regularized  set $\Ss_H$ does not have large enough dimension to go though all 
the homological constraints and 
so does not contribute to the invariant.
This is not true in
 higher genus because of the possible 
 existence of degenerate objects such as the 
nontransverse multiply covered fibers of Example~\ref{ex:rel2}. 

There are several points to note about the proof of the claim.  
Denote by $C_i^D, i=1,\dots,\ell$,
 the projections to $D$ of the higher level components  of $C$.
 The adjustments to the index coming from the multiplicities $\un d$ do not affect the argument, and so for simplicity we will
 ignore them, assuming in effect that $\un d=(1,\dots,1)$. 
 Then, the relative index is just ${\rm ind}_X\,\be$, the index of 
 genus zero $\be$ curves in $X$.

First note that,  by Lemma~\ref{le:kind1}, our perturbations are sufficient to
reduce the dimension of the moduli space of each $C_i^D$ to its index
(in $D$).   
Secondly,  we may arrange that the  $C_i^D$ 
intersect transversally in $D$ because  their 
 intersection pattern is a tree so that
two distinct components $C_i^D, C_j^D$ have at most one intersection point; 
cf. \cite[\S6.3]{MS}. 
Thirdly, because of the positivity assumption
 $c_1(L_D)(C_i^D)\ge 0$ for all $i$, the curves $C_i^D$  are regular when considered as curves in $X$,
 and (because they
all can be perturbed off $D$) their intersections are transverse in $X$ as well as in $D$.    

Now suppose that 
$\ell=1$ and $C$ has no components at level zero.  Then 
$C$ consists of single component in $Q$ whose projection to $D$ lies 
in some class $\be$,   and the 
claim holds  because 
  $c_1^D(\be)\le c_1^X(\be)$, so that $ {\rm ind}_D \,\be <  {\rm ind}_X\, \be$.
On the other hand, if $C$ has more than one component, it follows readily from the above remarks show that its index in the relative moduli space
is at most equal to the index of the stable map in $X$ with components $C_0, C_i^D$ and so is less than ${\rm ind}_X\,\be$. 
Hence if $C$ has higher levels it  does not lie in 
the top dimensional stratum and so cannot satisfy all the homological constraints in (\ref{eq:Gw0}).  

The failure of the above argument for \lq\lq negative" $D$  
is the reason why absolute 
 and relative GW invariants can differ in cases when one might think they should agree, e.g. if $\be\cdot D=0$; cf.  
 Proposition~\ref{prop:=0}.
 \end{rmk}

\subsection{First results on Gromov--Witten invariants}

We now explain some vanishing results that follow by dimensional arguments.
 As a warmup, we begin by considering the relative invariants of 
 $(\PP^n,\PP^{n-1})$.
 
\begin{lemma}\labell{le:000}  
 Let $X = \PP^n$  and  $D=\PP^{n-1}$ be the hyperplane. Denote by $\la$ the class of a line.
Suppose that $n>1$ and $d>0$  or $n=1$ and $d>1$.  Then
$$
\bla \;| b_1,\dots,b_r\bra^{\PP^n,\PP^{n-1}}_{0,d\la,\un d} = 0,\qquad \mbox{for any }
b_i\in H_*(\PP^{n-1}).
$$ 
\end{lemma}
\begin{proof}   
We  show that under the given conditions on $n$ and $d$  it is impossible to choose  $\un d$ and the $b_i$ so  that
$\oMm\,\!^{\PP^n,\PP^{n-1}}_{d\la,\un d}$ has formal dimension $0$.
Therefore the invariant vanishes by definition.

By equation~(\ref{eq:dim}) the formal  (complex) dimension of the  moduli space of genus $0$ stable maps through the relative constraints $b_1,\dots,b_r$ is 
$$
n + d(n+1) + r-3 - (d-r) + \de - rn,
$$
where $\de$ is half the sum of the degrees of the
 $b_i$.  
   Thus $0\le \de \le rn$.  
 Since $d-r\ge 0$ and $r>0$, we therefore need
$(d-r)n + n + 2r+\de= 3$, which implies $d=r$ and $n=r=1$. 
This contradicts our assumptions on $n$ and $d$.
\end{proof}

Our next result concerns manifolds that are fibrations with fiber $\PP^1$. These are important because, as we shall see in \S\ref{s:GWu} they can be used to understand the relation between the absolute and relative invariants. There are much more complete discussions of this topic in Maulik--Pandaripande \cite{MP}, Hu--Ruan \cite{HR} and T.-J. Li--Ruan 
\cite{TJLR}. 

  Let $L\to D$ be a complex line bundle and $Y: =  \PP(L\oplus \C)$
  be its projectivization.  It contains two divisors 
  that we call $D',D^+$, where
  $D': = 
\PP(\{0\}\oplus \C)$ is the copy of $D$  along the zero section of $L$ and  $D^+: = \PP(L\oplus \{0\})$ is the copy \lq\lq at infinity". Thus $D'$ has normal bundle $L$ while $D^+$ has normal bundle $L^*$.\footnote
{
There are many commonly used notations; we are using that  of \cite{Mcu} where the formula are applied to a divisor $D$ that is a blow up so that its normal bundle $L$ is considered as \lq\lq negative".} 

Given $\al\in H_2(D)$, we shall write $\al'$ for the corresponding element of $H_2(D')$, considered where appropriate as an element in $H_2(Y)$.
Let $f\in H_2(Y)$ be the fiber class of the projection $\pi:Y\to D$.

\begin{lemma}\labell{le:relD}  Let $\be \in H_2(Y)$ and $r,s\ge 0$.\SSS

\NI
{\rm 
(i)} If $\be\cdot D'\ge 0$ but $\be \ne \ell f$, then
$$
\bla a_1,\dots,a_s | b_1,\dots,b_r\bra^{Y,D^+}_{0,\be,\un d} = 0,
$$
for any $
b_i\in H_*(D^+)$ and any $a_j\in H_*(Y)$ that can be represented by cycles of the form $\pi^{-1}(a_j')$ where $a_j'\subset D$. \SSS

\NI {\rm (ii)}    If $\be = \ell f$ then the invariant is nonzero only if $\ell=1$,
and $s=0$.
\end{lemma}
\begin{proof} Write $\be = \ell f +\al'$ where $\ell\ge 0$.
Suppose first
that there are no absolute constraints and that
$\al' \ne 0$. 
Choose an $\om$-tame almost complex structure $J$ on $Y$ that is invariant under the $S^1$ action in the fibers of the ruled manifold $Y\to D$.  The fixed point set of $S^1$ consists of the divisors $D',D^+$,   and hence it fixes the relative constraints.  The localization principle\footnote{
\cite{MT} gives a proof for symplectic nonrelative invariants. The same arguments work for relative invariants.}
  implies that 
the only contributions to
 $\bla\; | b_1,\dots,b_r\bra^{Y,D^+}_{0,\be,\un d}$ come from $S^1$-invariant elements of the moduli space of $J$-holomorphic curves.   Such curves would have to consist of $r$ (possibly multiply covered) spheres in the fiber of $Y\to D$  together with a connected genus zero curve in $D'$.    
Since $\be\cdot D' = \ell + \al'\cdot D'\ge 0$, we must have
$\al'\cdot D' \ge -\ell$.  Hence $c_1^Y(\be) \ge
 \ell + 
 c_1^{D'}(\al) = : \ell+c'$ and so, because we always count curves of dimension zero, equation (\ref{eq:dim}) implies that
 $$
 n + \ell+c' + r-3 + \de -(\ell-r) \le rn,
 $$
 where $\de$ is the sum of half the dimensions of the constraints $b_i$.
 
 Now observe that we may regularize the moduli space $\oMm\,\!^{D'}_{\al'}$
 of $\al'$-curves (considered as curves  in $D'$)  by perturbations 
 that are tangent to $D'$ and hence are $S^1$ invariant. 
For the given relative invariant to be nonzero the  space 
of $S^1$-invariant stable maps  made up of such
partially regularized  $\al'$-curves plus some fibers going through the constraints must be nonempty. Hence there must be
 generic $\al'$-spheres in $D'$ with $r$ marked points 
that meet the fibers through $b_1,\dots,b_r$. In other words, we need
$
n-1 + c' + r-3 \ge r(n-1) - \de.
$
But this contradicts the inequality found above. 

Now suppose that there are some absolute constraints. Since 
these may be represented by cycles that are invariant under
 the $S^1$ action, $S^1$ still acts on the space of all $J$-holomorphic stable maps through the constraints.  
Hence the localization principle again says that the invariant is zero
unless this moduli space has fixed points.    These would consist
of  some fibers through the relative  constraints as well  as some of the absolute constraints, plus some $\al'$-curves in $D'$ meeting all these fibers and going through the other absolute constraints. 
Hence the  moduli space of partially regularized $\al'$-curves in $D'$ must meet all the $b_i$ as well as the classes $a_j'$.
A dimension count similar to that above  shows that this is impossible.
This proves (i).

In case (ii) we are counting $\ell$-fold coverings of the fiber $2$-spheres.
  The number of fibers through the constraints is finite and nonzero only if  the intersection product in $H_*(D)$ of all the
classes $b_i$ and $a_j'$ is nonzero and lies in $H_0(D)$.  
For each point in this intersection, we are then reduced to 
a fiber invariant of the form
$$
\bla \PP^1,\dots, \PP^1 | pt,\dots,pt\bra^{\PP^1,pt}_{0,\ell f,\un d},
$$
where the absolute constraints  are $s$ copies of 
$\PP^1 = a_i\cap \pi^{-1}(pt)$ and the relative ones are $r$ points.
  If $s>0$ such an invariant is nonzero only if the total number $s+r$ of  constraints  is $\le 3$. For then the indeterminancy of the absolute
  constraint(s) can be absorbed by reparametrizations so that there is a finite number of equivalence classes of  stable maps satisfying the given  conditions.
  A dimension count then shows that $\ell = 1$.   If $s=0$ 
  the invariant vanishes when $\ell>1$ by Lemma~\ref{le:000}.
\end{proof} 
 
 \begin{example}\labell{ex:relD}\rm  The vanishing result in 
 Lemma~\ref{le:relD} need not hold when $\be\cdot D'<0$.
  To see this, let $D=\PP^3\#2\ov{\PP}\,\!^3$ 
  be the  $2$-point blow up of $\PP^3$.  Denote the classes of $\PP^2$ and the two exceptional divisors by $H, E_1,E_2\in H_4(D)$;
  write $\la: = [\PP^1] = H\cdot H$ and $\eps_i: = E_i\cdot E_i$
  for the corresponding generators of $H_2(D)$. (Thus $\eps_i$ is represented by a line  in $E_i\cong \PP^2$.)
 Let $L$ be the line bundle over $D$ whose first Chern class $c_1^L$  is Poincar\'e dual to $H-E_1-E_2$.  (Thus $L$ is the normal bundle of $D$ in 
 the two point blow up of 
$\PP^4$ when $D$ is identified with the proper transform of a hyperplane through the two blown up points.) Let $Y = \PP(L\oplus \C)$ as above, with the obvious complex structure $J_0$.
  
  Consider the class $2\al' +f\in H_2(Y)$ where $\al' = \la'-\eps_1'-\eps_2'\in H_2(D')$. We claim that for all $\rho\in H_4(D)$
\begin{equation}\labell{eq:YD}
\bla\;|\rho\bra^{Y,D^+}_{0,2\al'+f, (1)} = \frac{\rho\cdot_D\al}4.
\end{equation}
where $\al\in H_2(D)$ corresponds to $\al'$ and $\cdot_D$ denotes the intersection product in $D$.
One could prove this by using algebro-geometric localization, but 
 we shall take another approach.  

As in Lemma \ref{le:relD}, it suffices to calculate the
contribution to this invariant 
from the moduli space $\Vv$ of $\C^*$-invariant
$J_0$-holomorphic $(2\al'+f)$-curves in $Y$, where $J_0$ denotes the obvious complex structure on $Y$.  
 The elements of $\Vv$ consist of doubly covered $\al'$-spheres in $D'$  plus a fiber.
We can regularize the space of doubly covered  
$\al'$-spheres in $D'$, obtaining a
virtual moduli cycle 
 $
 \oMm_{2\al'}(D')^{[vir]}
 $
  consisting of a finite number of embedded
 spheres $C'={\rm im\,} u'$ in $D'$ that each  have a rational weight and are  $(J_0,\nu)$-holomorphic, i.e. they  satisfy
a perturbed Cauchy--Riemann equation of the form $\ov{\p}_{J_0}
 u' = \nu(J_0,u')$.   The sum of these weights  is well known to be 
 $$
\bla\;\bra^{D'}_{0,2\al}=\bla\;\bra^{\PP^3\#2{\ov{\PP}\,\!^3}}_{0,2(\la - \eps_1-\eps_2)} = 1/8,
 $$  
see for example Gathmann~\cite[Ex~8.5]{Ga}.  
Therefore, we need to see that each element $u'\in \oMm_{2\al'}(D')^{[vir]}$ gives rise to two curves in $Y$ in class $2\al'+f$ and through $\rho$.  This may be done by using the  gluing methods developed by Zinger \cite[\S3]{Z}. We shall give a mroe elementary argument that  exploits the fact that $Y$ is an $S^2$-bundle over $D$.
 
The perturbing $1$-form $\nu$ used to construct 
$\oMm_{2\al'}(D')^{[vir]}$ is a (multivalued) section of the (orbi)bundle 
$\Ee(D)\to\Ww(D)$,
 where $\Ww(D)$ is a neighborhood of the space of $J_0$-holomorphic
 $2\al'$-spheres in the space of all stable maps into $D$ in class $2\al'$,  and   
$\Ee(D)\to \Ww(D)$ is the  (orbi)bundle
whose fiber at $[\Si,v]$ is the space of antiholomorphic $1$-forms on 
the nodal Riemann surface $\Si$ with values in $v^*(TD)$.
There is a similar bundle $\Ee(Y)\to  \Ww(Y)$ 
for stable maps in class $2\al'+f$.  The projection $\pi:Y\to D$
induces a map $\Ww(Y)\to \Ww(D)$, and also for each stable map
$[\Si,v]\in \Ww(Y)$ a surjective bundle map 
$\pi_*:v^*(TY) \to (\pi\circ v)^*(TD).$

Consider the space $\Ga^\nu$ of (multi)sections $\nu^Y$ of  $\Ee(Y)\to  \Ww(Y)$
that are $\nu$-compatible, in the sense that
$$
\pi_*(\nu^Y(v)) = \nu(\pi\circ v),\qquad \mbox {for all }v\in \Ww(Y).
$$
Then, for every $(J_0,\nu^Y)$-holomorphic map $v:\Si\to Y$,
the composite $\pi\circ v$ is a $(J_0,\nu)$-holomorphic map to $D$.
Therefore each $(J_0,\nu^Y)$-holomorphic sphere in $\Ww(Y)$ lies over
a $(J_0,\nu)$-holomorphic sphere in $\Ww(D)$, i.e. an element in $\oMm_{2\al'}(D')^{[vir]}.$ 
This reduces the problem of counting the $(J_0,\nu^Y)$-holomorphic
spheres in $Y$ through the constraint $\rho$ to a problem in the $4$-dimensional spaces $X_{u'}: = u'^*(\pi^{-1}(C'))$, where $C': = {\rm im\,}u'\cong S^2.$  But $X_{u'}\cong S^2\times S^2$ and the class $2\al'+f$ pulls back to $[pt \times S^2]$. 
Moreover, one can check directly that the allowed perturbations of $\nu^Y$ are sufficient to regularize the curves in $Y$; more precisely, for generic $\nu^Y$, every 
$(J_0,\nu^Y)$-holomorphic
sphere in $X_{u'}$ is regular both in $X_{u'}$ and also (because $u'$ is regular) in $Y$.  Hence the number of such spheres 
though a generic point  in $X_{u'}$ is $1$. 
Since $\rho\cdot_Y [\pi^{-1}(C')] = 2 \rho \cdot_D\al$, the  
result follows. 

\end{example}

%
\section{The decomposition formula}\labell{s:GWu}

The  main tool  in the theory of relative invariants
 is the decomposition rule
of  Li--Ruan~\cite{LR} and (in a slightly different version) of Ionel--Parker~\cite{IP}.  
So suppose that the manifold $(M,\Om)$ is the fiber sum of $(X,D, \om_X)$ with $(Y,D^+, \om_Y)$, where the divisors
 $(D, \om_X)$ and $ (D^+,\om_Y)$ are symplectomorphic with dual normal bundles. 
 This means that there is a (real) codimension $1$ hypersurface $S$ in $(M,\om)$ such that the orbits of its characteristic foliation form a  fibration of $S$ by circles  such that the quotient space can be identified with $(D,\om_X)$.
We will only consider cases in which $S$ divides $M$ into two components, one symplectomorphic to $X\less D$ and the other to $Y\less D^+$.
 
Let us first assume 
   that the 
absolute constraints can be represented by cycles in $M$ 
that do not intersect $S$, i.e. that
 $a_i\in H_*(X\less D), i\le q,$ and $a_i\in H_*(Y\less D^+), q<i\le p$.   
 For simplicity we will assume throughout  
  that the map $H_2(M)\to H_2(X\cup_DY)$ is injective.  (This hypothesis is satisfied whenever $H_1(D)=0$, and means that there are no rim tori (cf.
  \cite{IP}).)  
   Further, let $b_i, i\in I,$ be a basis for $H_*(D)= H_*(D;\Q)$ with
dual basis $b_i^*$ for $H_*(D)$.  Then
the  decomposition formula 
has the following shape:
\begin{eqnarray}\labell{eq:DF}
&&\bla a_1,\dots,a_p\bra^M_{g,\be}\; =\\\notag
&&\;  \sum_{\Ga,\un d,({i_1},\dots,{i_r})}  
n_{\Ga,\un d} \bla a_1,\dots,a_q\,|\, b_{i_1},\dots,b_{i_r}\bra^{\Ga_1,X,D}_{\be_1, \un d}
\bla a_{q+1},\dots,a_p\,|\, b_{i_1}^*,\dots,b_{i_r}^*\bra^{\Ga_2,Y,D^+}_{\be_2, \un d}. 
\end{eqnarray}
Here we  sum with rational weights $n_{\Ga,\un d}$ over all decompositions $\un d$ of $d$, all possible
 connected labelled trees $\Ga$, and all possible sets 
 ${i_1},\dots,{i_r}$ of relative constraints.  Each $\Ga$ describes a possible combinatorial structure for
 a stable map that glues to give
a $\be$-curve of genus $g$.  Thus
$\Ga$ is a disjoint union $\Ga_1\cup \Ga_2$, where the graph $\Ga_1$ (resp. $\Ga_2$) describes the part of the curve lying in some $X_k$ (resp. some $Y_\ell$) and $\be_1$ (resp. $\be_2$) is the part of its label that describes the  homology class.   Hence the pair $(\be_1,\be_2)$
runs through all decompositions such that the result of gluing the two curves in the prolongations 
$X_{k_1}$ and $Y_{k_2}$ 
along their intersections with the relative divisors 
gives a curve in class $\be.$
Moreover, there is a bijection between the labels $\{(d_i,b_i)\in \NN\times H_*(D)\}$
of the relative constraints in $\Ga_1$ and those
$\{(d_i,b_i^*)\in \NN\times H_*(D^+)\}$ in $\Ga_2$. (These labels are  called
relative \lq\lq tails"  in \cite{HLR}.)  $\Ga_i$ need not be connected; if it is not, we define  $\bla\dots|\dots\bra^{\Ga_i}$ to be the product of the invariants defined by its connected components.

\begin{rmk}\labell{rmk:DF}\rm (i)
In the genus zero case, each component of $\Ga_1$ is a tree and has at most one relative tail in common with  each component of $\Ga_2$.
In many cases one can show that one side  is connected, so
 that the other side  has $r$ components, one for each relative constraint.
\SSS

\NI(ii) The number $n_{\ga,\un d}$  is the reciprocal of the number of
automorphisms of the corresponding labelled graph.  We shall not be more specific here, because it will be $1$ in all cases where we shall need to know its value.  
\SSS

\NI (iii)  If absolute constraints $a_i\in H_*(M)$ cannot be represented on one side of $S$ or the other, one must represent  each $a_i$ by a cycle that decomposes as $a_i^X\cup a_i^Y$, where 
$a_{i}^X, a_i^Y$
are cycles in the closure of  the appropriate component of $M\less S$ with equal boundaries $a_i^X\cap S = a_i^Y\cap S$ on $S$ that are unions of $S^1$ fibers.  
Then in (\ref{eq:DF}) one must also sum over all
subsets $I\subset \{1,\dots,n\}$, where the $\Ga_1$-curve goes through the 
absolute constraints $a_i^X, i\in I$ and the $\Ga_2$-curve goes through $a_j^Y, j\notin I$.  If there are symmetries among these constraints (eg if $a_1^X=a_2^X$), one must adjust the multiplicities $n_\Ga$ accordingly.
\end{rmk}

The decomposition formula can be used to compare absolute and relative invariants.  Part (i) of the following result
was proved in Hu~\cite{Hu} and Gathmann~\cite{Ga} (in the projective case) under slightly weaker assumptions. 
Also see  Lai~\cite{Lai}.

\begin{prop}\labell{prop:g0}  
Suppose that $D$ is a divisor in $(X,\om)$ such that for some $\be\in H_2(X)$   
$$
c_1^X(\al)>(n-3)(k-2)\;\mbox{ if }\;n\ge 3,\qquad  c_1^X(\al)>0
\;\mbox{ if }\; n=2,
$$
for all  $\al\in H_2(D)$
with $0<\om(\al)<\om(\be)$ and $\al\cdot D = -k<-1$.
Then: \SSS

\NI {\rm (i)}
If $\be\cdot D = 0$, then  
 for any $a_i\in H_*(X\less D)$,
 $$
 \bla a_1,\dots,a_k|\;\bra^{X,D}_{0,\be}=\bla a_1,\dots,a_k\bra^{X}_{0,\be}.
 $$

\NI {\rm (ii)}  If $\be\cdot D > 0$, then  
 for any $a_i\in H_*(X\less D)$
 $$
 \bla a_1,\dots,a_k| D,\dots,D\;\bra^{X,D}_{0,\be,(1,\dots,1)}
 =\bla a_1,\dots,a_k\bra^{X}_{0,\be}.
 $$
 
\end{prop}
\begin{proof} (i) Let $L$ be the normal bundle of $D$.
Decompose $X$ into the sum of $(X,D)$ with the ruled manifold
 $(Y,D^+)$ where $Y:= \PP(L\oplus \C)$  and $D^+: = \PP(L\oplus \{0\})$, so that $D^+$ has normal bundle $L^*$.  We put the constraints into $X$.
There is one term in the
decomposition formula with $\Ga_2=\emptyset$ which contributes
$\bla a_1,\dots,a_k|\;\bra^{X,D}_{0,\be}$ to the absolute invariant.
Hence we need to see that there are no other terms.

Observe first that, even if $k=0$ it is impossible for there to be a nonzero term with $\Ga_1=\emptyset$. For then  the $\Ga_2$-invariant must be $\bla\;|\;\bra^{Y,D^+}_{0,\be}$, which vanishes  
by Lemma~\ref{le:relD} since $\be_2\cdot D'=\be\cdot D=0$.

 Therefore, suppose that there is 
a nonzero term with nontrivial $\Ga_1$-curve and with 
 the $\Ga_2$-curve  in the nonzero class $\be_2$.   
As in Lemma~\ref{le:relD} denote
 by $\be_{2,j}: = \ell_j f + \al_j'$  the classes of its connected components, where $\al'_j\in H_2(D')$.  
 If 
 $\al_j'\ne0$ then $\ell_j\ne 0$  because the whole curve is connected.  Hence Lemma~\ref{le:relD} implies that 
 $\be_{2,j}\cdot D' = 
 \ell_j +\al_j'\cdot D'<0$. Therefore
  $k_j: = -\al_j'\cdot D'\ge 2.$

We now show that this is ruled out by our dimensional hypothesis.
 Suppose that  the $\be_{2,j}$-curve has $r_j$ relative constraints with total (complex) dimension $\de_{bj}$. Then, because $c_1^Y(\al_j') =  c_1^X(\al_j') $,
 the dimensional condition
 $$
 n + 2\ell_j + c_1^X(\al_j') + r_j-3 +\de_{bj} - (\ell_j-r_j)-r_jn=0
 $$
 must hold.
Suppose now that $n\ge 3$.  Since $k_j>\ell_j\ge r_j\ge1$ the LHS decreases strictly if we replace $\ell_j$ by $r_j$, $\de_{aj}$ by $0$ and $c_1^X(\al_j')$ by $(n-3)(k_j-2)$.  Hence (dropping the subscripts $j$) we must have
$$
n + 3r-3 + (n-3)(k-2) -rn = (n-3)(k-r-1)< 0.
$$
 which is impossible.   A similar argument works when $n=2$. This proves (i).
 
 Now consider (ii).  There is a term in the decomposition formula 
for $\bla a_1,\dots,a_k\bra^{X}_{0,\be}$
in which a connected  $\Ga_1$-curve  in class $\be$  meeting $D$ transversally in $d=\be\cdot D$ arbitrary points is joined to
$d$ disjoint fibers in $Y$.   Again we must show that there are no other nonzero terms.
It follows from
 Lemma~\ref{le:relD} (ii) that any component in $\Ga_2$ consisting only of a fiber class $d_i f$ must have $d_i=1$. Hence, if there is another term one of the components of the $\Ga_2$ curve must lie in some class $\be_{2,j} = d_jf+\al_j'$ where $\al_j'\ne 0$.  But these can be ruled out just as in case (i).
  \end{proof}

\begin{rmk}\label{rmk:al}\rm (i)    As noted by Hu \cite{Hu}, this lemma applies when $D$ is the blow up of a point or of a curve $S$ with $c_1(S)\ge 0$.
\SSS


\NI (ii)
Maulik--Pandharipande~\cite{MP} show that the invariant
$\bla\;|b_1^*,\dots,b_r^*\bra^{Y,D^+}_{0,\ell f+\al', \un d}$ may be calculated  in terms of suitable descendent invariants of $D'$ 
in class $\al'$.  Hence in the above lemma  it also suffices  to restrict the assumption on 
$\al\cdot D$ to classes $\al$ for which the descendent GW invariants
in $D$ do not all vanish.
\end{rmk}

We now give some examples to illustrate the difference between 
the absolute and relative invariants.

It is easy to find examples with $\be\cdot D<0$ for which the two invariants are different since the relative invariant vanishes
by definition, while the absolute invariant might not.  For example, consider
$X = S^2\times S^2$ with $\al_1:=[S^2\times pt]$ and $\al_2:=
[pt\times S^2]$, and let $D$ be the antidiagonal in class $\al_1-\al_2$. Then 
$$
\bla pt|\;\bra^{X,D}_{0,\al_1}=0,\qquad \bla pt\bra^{X}_{0,\al_1}=1.
$$ 

A similar phenomenon may happen when $\be\cdot D=0$.  We first  
give a higher genus case and then in Proposition~\ref{prop:=0} a genus zero example.

\begin{example}\labell{ex:g1}\rm {\it Different absolute and relative genus one invariants when $\be\cdot D=0$:}\SSS

 Let $X = \PP(L\oplus \C)$ where 
$L\to T^2$ is a holomorphic line bundle of degree $1$
and put 
$D = \PP(L\oplus\{0\})$ the section of self-intersection $-1$.
Let $\be = D+f$, the class of the section $D^+: = \PP(\{0\}\oplus\C)$.
Then $\be\cdot D = 0$. The dimension of the
moduli space $\oMm_{1,1}(X,\be,J)$  of holomorphic tori in $X$ with one marked point is $2$,
and so there should a finite number of such tori through a generic point. 
We claim that:
\begin{equation}\labell{eq:g1}
\bla pt\bra^X_{1,\be}=2;\qquad \bla pt|\;\bra^{X,D}_{1,\be}=1.
\end{equation}

If $J$ is the obvious complex structure 
 on the ruled surface $X$   then
there is a $1$-parameter family of tori  $z\mapsto [\la\si(z):1]$ in class 
$\be$, where $\si$ is a holomorphic section of $L$ and $\la\in \C$.  All these tori 
go through the point $p$ where $D^+$ meets the fiber over the point $z_0$ where $\si$ vanishes.    
The Riemann--Roch theorem implies that these are all regular curves
 because they are embedded with self-intersection $>0$.
  There are infinitely many such tori through $p$, but
 just one through any other point $q\in X\less D$.  However there is also a reducible $\be$-torus through each point, the union of $D$ with a fiber of $X\to D$. 
 Thus this standard $J$ 
is {\it not} regular for tori in class $\be$.  In fact, one can show
that  $\bla pt\bra^X_{1,\be}=2$, i.e.
the family of reducible curves also contributes $1$ to this invariant. 
(This  can be checked by comparing with the Seiberg--Witten invariants; cf Li--Liu~\cite{LL}. Or one could make a gluing argument as in Example \ref{ex:relD}.)

To calculate the relative invariant $\bla pt|\;\bra^{X,D}_{1,\be}$,
 consider the relative moduli space $\oMm_{1,\be}(X,D;J)$ for the standard complex structure on $X$. (Note that
we must use  an almost structure 
 for which $D$ is holomorphic.) Then the prolongation $Q = \PP(\C\oplus L^*)$ can be identified with the complex manifold  $X$ in such a way
that $(D_0, D_\infty)$ correspond to $(D^+, D)$.  Hence the only holomorphic sections of $Q$ are tori through the point $z_D$ in $D_0$ lying over $z_0$. 
Therefore there is only {\it one} curve in  $\oMm_{1,\be}(X,J)$ with two levels;
it consists of the fiber over $z_0$ together with a torus in $Q$ through $z_D$.
Therefore, there is just one element of $\oMm_{1,\be}(X,D;J)$  that goes through a generic point of $X\less D$, and this element is regular.
Hence $\bla pt|\;\bra^{X,D}_{1,\be}=1$. This proves (\ref{eq:g1}).

Similarly, $\bla \;|pt\bra^{X,D^+}_{1,\be,(1)}=1$. The best way to see this is to take a generic complex structure for which $D^+$ is holomorphic.  Then because the  absolute GW invariant is $2$ there is one more holomorphic torus through each point of $D^+$.  This set of tori form the main stratum of the moduli space.   The curve $D^+$ itself does not appear  and, as before, there is only one section in $Q$. Since this may well not go through  the given point constraint on $D_{1\infty}$ it does not contribute to the invariant.

On the other hand $\bla pt| D^+\bra^{X,D^+}_{1,\be,(1)}=2$, since if $J$ is a generic complex structure for which $D^+$ is holomorphic there are precisely two holomorphic tori in class $\be$ through any point not on $D^+$ that each meet $D^+$ transversally.

We now show that {\it these results are  consistent with the decomposition formula.}  Let us calculate
$
\bla pt\bra^{X}_{1,\be}
$
 by thinking of $X$ as the connected sum of $(X,D)$ with the ruled surface $(Y,D^+)$, defined as in Proposition~\ref{prop:g0}.  Since  $(Y,D^+)$ 
can be identified with $(X,D^+)$ we have
$
\bla \;| pt\bra^{Y,D^+}_{1,\be,(1)}=1.
$
Therefore, if we put the point constraint into $X$, besides the term
$$
\bla pt|\;\bra^{X,D}_{1,\be}=1
$$ with $\Ga_2=\emptyset$,
there also is a term in the
decomposition formula in which
the curve in $(X,D)$ is  the sphere 
fiber through $pt$, and the curve in the ruled surface $Y$ is 
 a torus in class $\be$ meeting $D^+$ in a point constraint.   Thus this contribution is the product
$$
\bla pt|D\bra^{X,D}_{0,f,(1)}\bla\; | pt\bra^{Y,D^+}_{1,\be,(1)}=1.
$$
Note that here the constraints along the divisor $D\equiv D^+$ are dual, as required. Thus
$$
\bla pt\bra^{X}_{1,\be}\; =\; \bla pt|\;\bra^{X,D}_{1,\be} \;+\; 
\bla pt|D\bra^{X,D}_{0,f,(1)}\bla\; | pt\bra^{Y,D^+}_{1,\be,(1)}\;=2.
$$

One can also apply the decomposition formula with the point 
constraint in $Y$.  Then there is just one nonzero term, and we find
$$
\bla pt\bra^X_{1,\be}=\bla  \;| pt\bra^{X,D}_{0,f,(1)}\;\bla pt | D^+\bra^{Y,D^+}_{1,\be,(1)} = 2,
$$
as before.\hfill$\Box$
\end{example}

\subsection{A genus zero example}

In this section we give a genus zero example where  
$\be\cdot D = 0$ but the relative and absolute invariants are different.
Throughout 
$X$ is  the blow up of $\PP^4$  at two points, and $D$ is the proper transform of a hyperplane through these two points.  Thus $D = \PP^3\#2{\ov \PP}\,\!^3$  as in 
Example~\ref{ex:relD}.   As a homology class in $X$, $D = H^X-E_1^X-E_2^X$ where $H^X$ is the hyperplane in $X$ and $E_j^X$ are the exceptional divisors.
We define $\la$ to be the class of the line in $\PP^4$ and  $\eps_j$ to be the class
of the line in $E_j$, considered where appropriate
as elements of $H_2(D)$ or $H_2(X)$.   We shall use the self-dual basis 
\begin{equation}\labell{eq:bas}
pt, \;\;\la,\;\; \eps_1,\;\;\eps_2, \;\;\la^*: = \pi, \;\;\eps_1^*, \;\;\eps_2^*,\;\; pt^*=D
\end{equation}
 for $H_*(D)$.  Thus $\pi: = D\cap H^X$ and
 $\eps_j^*= -D\cap E_j^X$ are disjoint (and equal to the classes $H, -E_j$ in the notation of Example~\ref{ex:relD}).
 
  We need two auxiliary lemmas.  The first is due to Gathmann.
It compares certain Gromov--Witten invariants in $X$ and $\PP^4$, and
 expresses the fact that these invariants are enumerative, i.e. they count what one expects them to count.
 
 \begin{lemma}\labell{le:X}  Let $X=\PP^4\#2{\ov \PP}\,\!^4$ be as above and $a_1,\dots,a_m\in H_{<8}(\PP^4)\subset H_{<8}(X)$.
 
 \NI {\rm (i)} For any $k>0$,  $\bla a_1,\dots,a_m\bra^X_{0,k\la} =  
 \bla a_1,\dots,a_m\bra^{\PP^4}_{0,k\la}.$\SSS
 
 \NI {\rm (ii)}   For any $k>0$ and $i=1,2$,  $\bla a_1,\dots,a_m\bra^X_{0,k\la-\eps_i} =  
 \bla pt, a_1,\dots,a_m\bra^{\PP^4}_{0,k\la}.$\SSS
 
 \NI {\rm (iii)}   For any $k>0$,  $\bla a_1,\dots,a_m\bra^X_{0,k\la-\eps_1-\eps_2} =  
 \bla pt, pt, a_1,\dots,a_m\bra^{\PP^4}_{0,k\la}.$
 \end{lemma}
 
 The first part of the next lemma is also well known; we give the proof for completeness.

 \begin{lemma}\labell{le:P4} 
 Let $X,D$ be as above.  Then
 \SSS
 
 \NI
 {\rm (i)}  
 $\bla pt,pt,\la,\pi,\pi,\pi\bra^{\PP^4}_{0,2\la} =4$.\SSS
 
 \NI {\rm (ii)} $\bla pt,pt,\pi,\pi,\pi\;|\;\la,D\bra^{X,D}_{0,2\la,(1,1)} =8$.
  \end{lemma}
 
 \begin{proof}  (i)  Put the constraints $pt,pt, {\it line}$ in general position and let $P$ be the $3$ space containing them.  Then every conic (i.e. degree $2$ 
 holomorphic curve) through these constraints intersects $P$ in three points and so must lie entirely in $P$.  It follows readily that 
 $$
 \bla pt,pt,\la,\pi,\pi,\pi\bra^{\PP^4}_{0,2\la}  =
 \bla pt,pt,\la,\la,\la,\la\bra^{\PP^3}_{0,2\la}.
 $$
 Now calculate the invariant by applying the splitting rule 
 (cf. \cite[Ch~7.5]{MS}) to
 $$
 I= \bla pt,pt,\pi,\pi;\la,\la,\la\bra^{\PP^3,\{1,2,3,4\}}_{0,2\la}, 
 $$
 where the superscript $\{1,2,3,4\}$ means that the cross ratio of the first four marked points is fixed and we have dropped the subscript $H$.  If the constraints are split as $pt,pt$ and $\pi,\pi$ then we find
\begin{eqnarray*}
 I &=& \bla pt,pt,\la,\la,\la,\la\bra^{\PP^3}_{0,2\la}\;
  \bla \la^*,\pi,\pi\bra^{\PP^3}_{0,0} +
   \bla pt,pt,\pi\bra^{\PP^3}_{0,\la}\;
  \bla\pi^*,\la,\la,\la\bra^{\PP^3}_{0,\la}\\
  &=& \bla pt,pt,\la,\la,\la,\la\bra^{\PP^3}_{0,2\la} + 2,
\end{eqnarray*}
because $\la^*=\pi$ and $\bla\la,\la,\la,\la\bra^{\PP^3}_{0,\la}=2$ by \cite[Example~7.1.16]{MS}.
On the other hand if they are split with $pt, \pi$ in each factor then we get
 $$
 I = 6
 \bla pt,\pi, \pi,\la,\la\bra^{\PP^3}_{0,\la}\;
  \bla pt,\pi, \pi^*,\la\bra^{\PP^3}_{0,\la} = 6,
  $$
  where the second equality uses the fact that 
  $ \bla pt,\la,\la\bra^{\PP^3}_{0,\la}=1$ by 
  \cite[Example~7.1.14]{MS}.
This proves (i). 
   
  Now consider (ii).  We first claim that 
$$
\bla pt,pt,\pi,\pi,\pi\;|\;\la,D\bra^{X,D}_{2\la,(1,1)}=
\bla pt,pt,\pi,\pi,\pi,\la,D\bra^{X}_{2\la}.
$$
This does not immediately follow from
Proposition~\ref{prop:g0}(ii)
since $D$ does not satisfy the conditions of this lemma for all classes $\al$.
However, all we need is that
 $c_1^X$  is sufficiently positive on
the classes  $\al$ that could appear in a decomposition $(\be_1,\be_{2})$
of the class $2\la$.     But this condition is satisfied because, if we write
$\be_{2} = \ell f + s'\la' - m_{1}'\eps_1'-m_ {2}'\eps_2'$, we must have $m_1'=m_2'=0$.
(Otherwise there would be a connected component in $\Ga_1$ in some class $p\eps_j$ for $p\ge 1$.  But such a class cannot be controlled by 
the available absolute constraints since $\pi$ can be represented 
by a pseudocycle that is disjoint from the $\eps_j$-curves.\footnote
{
One must be careful here since there are nonzero invariants in class $\eps_j$.
For example, if  $\eps_j, \eps^*_j$ denote the line and $2$-plane in the exceptional divisor $E_j^X$ in $X$
$\bla \eps^*_j,\eps^*_j\;|\;\eps_j\bra^{X,D}_{0,\eps_j}=
\bla \eps_j,\eps_j,pt\bra^{E_j^X}_{0,\eps_j}=1
$.})
Therefore
\begin{eqnarray*}
\bla pt,pt,\pi,\pi,\pi|\la,D\bra^{X,D}_{2\la,(1,1)}&=&
\bla pt,pt,\pi,\pi,\pi,\la,D\bra^{X}_{2\la}\\
&=&
2\bla pt,pt,\pi,\pi,\pi,\la\bra^{X}_{2\la},
\end{eqnarray*}
   where the last equality holds by the divisor axiom 
   and the fact that $2\la\cdot D = 2$. Now apply Lemma~\ref{le:X} to reduce this to an invariant in $\PP^4$ and use (i).
 \end{proof}

We now give an explicit example in genus zero where the absolute and relative invariants differ. 

\begin{prop}\labell{prop:=0} Let $(X,D)$ and $\pi$ be as in the previous lemma,
and let 
$\be=4\la -2\eps_1-2\eps_2$.  Then
$$
\bla pt,pt,\pi,\pi,\pi\bra^{X}_{\be} \ne 
\bla  pt,pt,\pi,\pi,\pi|\;\bra^{X,D}_{\be}.
$$
\end{prop}
\begin{proof} The difference between these invariants is 
$$
\sum_{\be_2\ne0, b_i, a_j,\un d}
\bla pt,pt,a_1,\dots, a_q|b_1,\dots,b_r\bra^{\Ga_1,X,D}_{\be_1,\un d}
\bla a_1',\dots, a_{q'}'\;|b_1^*,\dots,b_r^*\bra^{\Ga_2,Y,D^+}_{\be_2,\un d}
$$
where we sum over all possible relative constraints and all placings of the
absolute constraints $\pi, \pi,\pi$ (which we decompose as in 
Lemma~\ref{le:relD}). Thus $q+q' = 3$.   

There is
one nonzero term in this sum with $\be_1=2\la$ and $\be_2=2\al'+2f$.
Here we take the $\Ga_1$ curve to be  a connected $2\la$
 curve through all the
absolute constraints and the relative constraints $b_1:=D, b_2:=\la$,
and the $\Ga_2$ curve to be disconnected, with one component equal to a fiber through a point in $b_1^*=pt$ and the other a sphere 
in class $2\al'+f$ and through
the relative constraint $b_2^*=:\pi$.  The $\Ga_2$ factor equals $1/4$ by
Example~\ref{ex:relD}.  
Further,
by Lemma~\ref{le:P4}
$$
\bla pt,pt,\pi,\pi,\pi\,|\,\la,D\bra^{X,D}_{0,2\la,(1,1)} =8.
$$
Therefore this term contributes $2$ to the sum.

Since curves in class $2\al$ do not meet $D\cap E_j$, $$
\bla pt,pt,\pi,\pi,\pi\,|\,\eps_j,D\bra^{X,D}_{0,2\la,(1,1)}=0, \quad j=1,2,
$$
where we have used the notation of equation (\ref{eq:bas}).
  Moreover, although
 $$
\bla pt,pt,\pi,\pi\,|\,pt,D\bra^{X,D}_{0,2\la,(1,1)}=
\bla pt,pt,\pi,\pi, pt\bra^{X}_{0,2\la}=1,
$$
the corresponding $\Ga_2$ invariant vanishes.  For, this would have to
consist of a fiber together with a connected curve in class $f+2\al'$.
This is nonzero only if
there is a $\C^*$ invariant representative of $f+2\al'$
 through the remaining absolute constraint. But 
 we can arrange that there are no such curves:
 by Remark \ref{rmk:DF} (iii) 
the constraint has the form $\pi^{-1}(\la')$, where $\la'$ is a line in $D'$ that can certainly be chosen disjoint from the $\al'$ curve.
A similar argument rules out the case when $\Ga_1$ consists of
two components each in class $\la$, since such $\Ga_1$ cannot go through all the absolute constraints.\footnote
{
It is possible to use up two absolute constraints because
 $$
\bla pt, \pi,\pi\,|\,\pi\bra^{X,D}_{0,\la} = 
\bla pt, \pi,\pi,\pi\bra^{\PP^4}_{0,\la}=
\bla pt, \pi,\la,\la\bra^{\PP^3}_{0,\la}=1.
$$
}
Therefore there are no other contributions with
$\Ga_1$ a connected curve in class $2\la$.

We now show that there are no other nonzero terms in the sum. To
see this,  
consider the possible choices for $\be_2$.  By Lemma~\ref{le:relD},
each connected component of a nonzero $\Ga_2$  factor must either lie in class $f$ or in some class $\be_{2,i}$ such that $\be_{2,i}\cdot D'<0$. 
Since $\be\cdot D < 0$ there must be at least one term of the latter form.

Consider such a term and
 write  $\be_{2,i} = k_i f + s_i'\la' - m_{i1}'\eps_1'-m_ {i2}'\eps_2'$.  Then since $k_i\ge 1$ we need  
$s_i'-m_{i1}'-m_{i2}'\le -2$.
Since the class $s_i'\la- m_{i1}'\eps_1-m_{i2}'\eps_2$ must be representable
in $D$ we need either $s_i'>0$ or  $s_i'=0$ and  $m_{i1}'\le 0$.
 But the latter possibility makes $s_i'-m_{i1}'-m_{i2}'\ge0$.
Therefore $s_i'>0$.   

We next claim:

\begin{quote}
{\it if the class $s\la-m_1\eps_1-m_2\eps_2$ with $s>0$ is represented by a holomorphic sphere $S$ in $D$ then either $m_1=m_2=s$ or
$0\le m_j$ and $m_1+m_2\le s$.}
\end{quote}

\NI {\it Proof of Claim:} In the notation of Example \ref{ex:relD},
 $$
 (s\la-m_1\eps_1-m_2\eps_2)\cdot_D(H-E_1-E_2) = s-m_1-m_2.
 $$
If this is negative then $S$ lies entirely in the plane
$H-E_1-E_2\cong \PP^2\#2{\ov \PP}\,\!^2$.  Hence it must either be a multiple of the unique curve in class $\al = \la-\eps_1-\eps_2$ or have nonnegative intersection with it. The claim follows.\SSS

The connected component $\be_{2,i}$ might be a union of spheres,
but, by the claim, the homology class of each such sphere
is  $s_i\al, m\eps_j $ or $s\la-m_1\eps_1-m_2\eps_2$ with $s\ge m_1+m_2\ge0$.
Since the classes $s_i\al$ are the only ones with negative intersection with $D'$ there must be at least one of these. Moreover, we need $s_i\ge 2$ for the intersection to be negative and also $\sum_i s_i\le 3$ since the $\Ga_1$ component 
goes through the two point constraints.  

If  $\sum_i s_i=3$,
then the $\Ga_1$ component must be
$$
\bla pt, pt\,|\,D\bra^{X,D}_{0,\la}.
$$
Therefore the $\Ga_2$ component is connected and must go through the
absolute constraints.  As noted above, these have the form $\pi^{-1}(\la')$ where we may take $\la'\subset D'$ to be
 disjoint from the exceptional divisors $E_1', E_2'$.
Further $\be_2=f+3\al+\eps_1+\eps_2$. Therefore its only holomorphic and $\C^*$ invariant representatives consist of the union of a fiber with a triply covered $\al$-curve and some curves in $E_1'\cup E_2'$.
Since these do not meet the absolute constraints, there are no contributions of this form.

Therefore $\Ga_2$ has just one nonfiber component, and this must lie in class $kf+2\al$. Hence $\be_1=2\la$. But we have already seen that there is only one term of this form.  This completes the proof. 
\end{proof}

%


%
 
\begin{rmk}\rm One could calculate $\bla pt,pt,\pi,\pi,\pi\bra^{X}_{\be}$
by Gathmann's algorithm, but the argument is quite complicated and is unnecessary for the present purpose.  Note that the reason why
the two invariants differ in the case considered here is precisely the same as the reason why many of the invariants considered by  Gathmann are not enumerative, namely the class $\be$ can be represented by stable maps 
with components in $D$ that are regular in $D$ but not in $X$.
\end{rmk}

\end{document}